\documentclass[12 points]{article}
\usepackage{amssymb}
\usepackage{eucal}
\usepackage{amsmath}
\usepackage{amsthm}
\usepackage{graphicx}
\usepackage{float}
\usepackage{extraipa}
\usepackage{framed}
\usepackage{pbox}

\usepackage{amsfonts}
\usepackage{mathabx} 
\usepackage{enumitem}

\newcommand{\wt}{\widetilde}
\newcommand{\R}{{\mathbb  R}}  \numberwithin{equation}{section} \newtheorem{thm}{\bf
Theorem}[section]
 \newtheorem{prop}[thm]{\bf Proposition} 
  \theoremstyle{remark}

\begin{document}

\title{\huge\bf Steepest descent algorithm on orthogonal Stiefel manifolds} 
 \author{Petre Birtea, Ioan Ca\c su and Dan Com\u{a}nescu \\ {\small Department of Mathematics, West University of Timi\c soara}\\ {\small Bd. V.
P\^ arvan,
No 4, 300223 Timi\c soara, Rom\^ania}\\ {\small petre.birtea@e-uvt.ro, ioan.casu@e-uvt.ro, dan.comanescu@e-uvt.ro}\\ }
\date{ }

 \maketitle
 
 \begin{abstract}
 Considering orthogonal Stiefel manifolds as constraint manifolds, we give an explicit description of a set of local coordinates that also generate a basis for the tangent space in any point of the orthogonal Stiefel manifolds. We show how this construction depends on the choice of a  submatrix of full rank.
 Embedding a gradient vector field on an orthogonal  Stiefel manifold in the ambient space, we give explicit necessary and sufficient conditions for a critical point of a cost function defined on such manifolds. We explicitly describe the steepest descent algorithm on the orthogonal Stiefel manifold using the ambient coordinates and not the local coordinates of the manifold. We point out the dependence of the recurrence sequence that defines the algorithm on the choice of a full rank submatrix. We illustrate the algorithm in the case of Brockett cost functions.
 \end{abstract}



\section{Introduction}

In Section 2 we construct an atlas for the orthogonal Stiefel manifolds
$St_p^n=\{U\in \mathcal{M}_{n\times p}(\R) \,|\,U^TU=\mathbb{I}_p\}$ following an idea from \cite{fraikin}. The local charts that we introduce crucially depend on the choice of a full rank submatrix of the elements in the orthogonal Stiefel manifolds. More precisely, for $U\in St^n_p$,  if $I_p$ is the set of row indexes that form a full rank submatrix of the matrix $U$, then we define the vector subspace $W_{I_p}=\left\{\Omega=\left[\omega_{ij}\right]\in \text{Skew}_{n\times n}(\R)\,\left|\, \omega_{ij}=0\,\,\text{for all}\,\,i\notin I_p\,\hbox{and}\, j\notin I_p \right.\right\}$. The local charts are defined by $\varphi_U:W_{I_p}\rightarrow St_p^n,\,\,\, \varphi_U(\Omega):=\mathcal{C}(\Omega)U$,
 where $\mathcal{C}(\Omega)=(\mathbb{I}_n+\Omega)(\mathbb{I}_n-\Omega)^{-1}$ is the Cayley transform.  These local charts provide us with a basis for the tangent spaces to the orthogonal Stiefel manifolds.

In Section 3 we present necessary and sufficient conditions for a critical point of a cost function defined on an orthogonal Stiefel manifold using the embedded vector field method \cite{Birtea-Comanescu-Hessian}, \cite{birtea-comanescu}, and \cite{birtea-comanescu-5-electron}. We describe necessary and sufficient conditions for critical points in the case of Procrustes and Penrose regression cost functions, sums of heterogeneous quadratic forms, and Brockett cost functions. We also discuss our findings in comparison with existing results in the literature \cite{wen}, \cite{chu}, and \cite{bolla}.

In the last section we give an explicit description of the steepest descent algorithm taking into account the specificity of the orthogonal Stiefel manifold. On a general Riemannian manifold $(S,{\bf g}_S)$ the iterative scheme of steepest descent algorithm is given by
\begin{equation*}
x_{k+1}={\mathcal{R}}_{x_k}(-\lambda_k\nabla_{{\bf g}_{_S}}\wt{G}(x_k)),
\end{equation*}
where $\widetilde{G}:S\rightarrow \R$ is the cost function that we want to minimize, $\mathcal{R}:TS\rightarrow S$ is a smooth retraction and $\lambda_k\in \R$ is a chosen step length. For the case of orthogonal Stiefel manifolds, we write the vector $\nabla_{{\bf g}_{_S}}\wt{G}(x_k)$ as a vector in the ambient space $T_{x_k}M$ using the embedded gradient vector field $\partial G(x_k)$ (see \cite{birtea-comanescu} and \cite{birtea-comanescu-5-electron}), i.e., $\nabla_{{\bf g}_{_S}}\wt{G}(x_k)=\partial G(x_k)$. The explicit description of the vector $\partial G(U_k)$ on an orthogonal Stiefel manifold depends on the chosen basis for $T_{U_k}St^n_p$, which in turn depends on the chosen full rank submatrix of $U_k$. In order to write the vector $-\lambda_k\partial{G}(U_k)$ as a tangent vector in $T_{U_k}St^n_p$ we have to solve the matrix equation $-\lambda_k\partial G(U_k)=\Omega_kU_k$ for the unknown skew-symmetric matrix $\Omega_k\in W_{I_p(U_k)}$. Once we have solved for $\Omega_k$, we construct the next term of the iterative sequence as $U_{k+1}=\left(\mathbb{I}_n+\frac{1}{2}\Omega_k\right)\left(\mathbb{I}_n-\frac{1}{2}\Omega_k\right)^{-1}U_k.$
Moreover, using an appropriate permutation matrix for each step of the algorithm we   give an explicit elegant solution of the matrix equation  $-\lambda_k\partial G(U_k)=\Omega_kU_k$, which makes the steepest descent algorithm more implementable. 
We exemplify the form of the steepest descent algorithm that we have constructed on orthogonal Stiefel manifolds for the case of two Brockett cost functions. 

Another method to construct numerical algorithms in the presence of orthogonal constraints of the Stiefel manifolds is presented in \cite{kanamori} and the authors use the so called Alternating Direction Method of Multipliers (ADMM), see \cite{boyd} and \cite{zhang} for a general description. ADMM is a variant of the Augmented Lagrangian Method of Multipliers introduced in \cite{glowinski-marrocco}, see also \cite{glowinski} for a historical presentation of the method. A deep convergence result for the extension of ADMM to multi-block convex minimization problems is proved in \cite{chen}.

\section{Local charts on the orthogonal Stiefel manifolds}

In this section we will construct a local chart around every point $U\in St_p^n$ and a basis for the tangent space $T_USt_p^n$. We will follow the idea presented in \cite{fraikin}, where the authors have constructed a local chart around points closed to $\left[ \mathbb{I}_p\,\,\mathbb{O}_{(n-p)\times p}\right]^T\in St_p^n$. This corresponds to the particular situation when the full rank submatrix of the point $U\in St_p^n$ is formed with the first $p$ rows. For a general $U\in St_p^n$ a modification of the construction presented in \cite{fraikin} is necessary.

Let $U\in St_p^n$ and $1\leq i_1<...<i_p\leq n$ be the indexes of the rows that form a full rank submatrix $\bar{U}$ of $U$. We denote $I_p=\{i_1,...,i_p\}$. Let $\text{Skew}_{n\times n}(\R)$ be the $\frac{n(n-1)}{2}$-dimensional vectorial space of the real skew-symmetric $n\times n$ matrices.
We introduce the following $np-\frac{p(p+1)}{2}$-dimensional vectorial subspace of $\text{Skew}_{n\times n}(\R)$:
$$W_{I_p}:=\left\{\Omega\in \text{Skew}_{n\times n}(\R)\,\left|\,\Omega=  \sum_
{\substack{
i<j;\,
   i,j\in I_p 
  }}
\omega_{ij}\,({\bf e}_i\otimes {\bf e}_j-{\bf e}_j\otimes {\bf e}_i)
+\sum_
{\substack{
   i\in I_p;\,
j\notin I_p
  }}
\omega_{ij}\,({\bf e}_i\otimes {\bf e}_j-{\bf e}_j\otimes {\bf e}_i)\right.\right\},$$
where the vectors ${\bf e}_1$, ...   ,${\bf e}_n$ form the canonical basis in the Euclidean space $\R^n$. The $n\times n$ matrix ${\bf e}_i\otimes {\bf e}_j$ has $1$ on the $i$-th row and $j$-th column and $0$ on all remaining positions. 
An equivalent description of the vectorial subspace $W_{I_p}$ is given by 
$$W_{I_p}=\left\{\Omega=\left[\omega_{ij}\right]\in \text{Skew}_{n\times n}(\R)\,\left|\, \omega_{ij}=0\,\,\text{for all}\,\,i\notin I_p\,\hbox{and}\, j\notin I_p \right.\right\}.$$
Around the point $U\in St_p^n$ chosen above we construct the local chart 
\begin{equation}\label{charts}
\varphi_U:W_{I_p}\rightarrow St_p^n,\,\,\, \varphi_U(\Omega):=\mathcal{C}(\Omega)U,
\end{equation}
 where $\mathcal{C}(\Omega)=(\mathbb{I}_n+\Omega)(\mathbb{I}_n-\Omega)^{-1}$ is the Cayley transform.
We notice that $\varphi_U$ is a smooth map with $\varphi_U({\bf 0})=U$. In order to prove that $\varphi_U$ is a local chart it is sufficient to prove that $\varphi_U$ is locally injective around ${\bf 0}\in W_{I_p}$, which in turn is implied by injectivity of the linear map $\displaystyle d\varphi_U({\bf 0})$. The later condition is equivalent with the vectors $\displaystyle\frac{\partial \varphi_U}{\partial \omega_{ij}}({\bf 0})$ being linearly independent.

In what follows we use the notation:
$$\Lambda_{ij}:={\bf e}_i\otimes {\bf e}_j - {\bf e}_j\otimes {\bf e}_i\in \mathcal{M}_{n\times n}(\R),$$
for any $1\leq i,j\leq n$, $i\ne j$.

An easy computation shows that (see \cite{petersen-pedersen}\footnote{We have used the following formula for the derivative of the inverse of a matrix: $\frac{\partial A^{-1}}{\partial x}=-A^{-1}\frac{\partial A}{\partial x}A^{-1}$})
\begin{align*}
\frac{\partial\varphi_U}{\partial \omega_{ij}}(\Omega)=&\left(\Lambda_{ij}(\mathbb{I}_n-\Omega)^{-1}+(\mathbb{I}_n+\Omega)(\mathbb{I}_n-\Omega)^{-1}\Lambda_{ij}(\mathbb{I}_n-\Omega)^{-1}\right)U\\
=&\left(\mathbb{I}_n+(\mathbb{I}_n+\Omega)(\mathbb{I}_n-\Omega)^{-1}\right)\Lambda_{ij}(\mathbb{I}_n-\Omega)^{-1}U\\
=&\left((\mathbb{I}_n-\Omega)(\mathbb{I}_n-\Omega)^{-1}+(\mathbb{I}_n+\Omega)(\mathbb{I}_n-\Omega)^{-1}\right)\Lambda_{ij}(\mathbb{I}_n-\Omega)^{-1}U\\
=& 2(\mathbb{I}_n-\Omega)^{-1}\Lambda_{ij}(\mathbb{I}_n-\Omega)^{-1}U.
\end{align*}
Consequently, we have
$$\frac{\partial\varphi_U}{\partial \omega_{ij}}({\bf 0})=2\Lambda_{ij}U.$$
For proving the linear independence of the vectors $\displaystyle\frac{\partial \varphi_U}{\partial \omega_{ij}}({\bf 0})$, we consider the equation
$$\sum_{\substack{i<j;\,i,j\in I_p}}\alpha_{ij}\frac{\partial \varphi_U}{\partial \omega_{ij}}({\bf 0})+\sum_{\substack{r\in I_p;\,s\notin I_p}}\beta_{rs}\frac{\partial \varphi_U}{\partial \omega_{rs}}({\bf 0})=\mathbb{O}_{n\times p},$$
which is equivalent with\footnote{The vectors ${\bf f}_1$, ...   ,${\bf f}_p$ form the canonical basis in the Euclidean space $\R^p$. We use the rule for matrix multiplication 
$\left({\bf u}\otimes {\bf v}_{\oplus}\right)\cdot \left({\bf v}_{\odot}\otimes {\bf w}\right)=\delta_{\oplus,\odot}{\bf u}\otimes {\bf w},$
where ${\bf v}_{\oplus}$ and ${\bf v}_{\odot}$ belong to the same vectorial space.}
\begin{align*}
\mathbb{O}_{n\times p}=&\sum_{\substack{i<j;\,i,j\in I_p}}\alpha_{ij}\Lambda_{ij}U+\sum_{\substack{r\in I_p;\,s\notin I_p}}\beta_{rs}\Lambda_{rs}U\\
=&\sum_{\substack{i<j;\,i,j\in I_p}}\alpha_{ij}\Lambda_{ij}\left( \sum_
{\substack{
   k\notin I_p \\
b\in \{1,...,p\}
  }}
u_{kb}\,{\bf e}_k\otimes {\bf f}_b+
\sum_
{\substack{
   q\in I_p \\
a\in \{1,...,p\}
  }}
u_{qa}\,{\bf e}_q\otimes {\bf f}_a\right)+\\
+&\sum_{\substack{r\in I_p;\,s\notin I_p}}\beta_{rs}\Lambda_{rs}\left( \sum_
{\substack{
   k\notin I_p \\
b\in \{1,...,p\}
  }}
u_{kb}\,{\bf e}_k\otimes {\bf f}_b+
\sum_
{\substack{
   q\in I_p \\
a\in \{1,...,p\}
  }}
u_{qa}\,{\bf e}_q\otimes {\bf f}_a\right)\\
=&\sum_{\substack{i<j;\,i,j\in I_p}}\alpha_{ij}\left({\bf e}_i\otimes {\bf e}_j - {\bf e}_j\otimes {\bf e}_i\right)\left( \sum_
{\substack{
   k\notin I_p \\
b\in \{1,...,p\}
  }}
u_{kb}\,{\bf e}_k\otimes {\bf f}_b+
\sum_
{\substack{
   q\in I_p \\
a\in \{1,...,p\}
  }}
u_{qa}\,{\bf e}_q\otimes {\bf f}_a\right)+\\
+&\sum_{\substack{r\in I_p;\,s\notin I_p}}\beta_{rs}\left({\bf e}_r\otimes {\bf e}_s - {\bf e}_s\otimes {\bf e}_r\right)\left( \sum_
{\substack{
   k\notin I_p \\
b\in \{1,...,p\}
  }}
u_{kb}\,{\bf e}_k\otimes {\bf f}_b+
\sum_
{\substack{
   q\in I_p \\
a\in \{1,...,p\}
  }}
u_{qa}\,{\bf e}_q\otimes {\bf f}_a\right)\\
=&\sum_{\substack{i<j;\,i,j\in I_p\\k\notin I_p;\,b\in \{1,\dots,p\}}}\alpha_{ij}u_{kb}\left(\delta_{jk}{\bf e}_i\otimes {\bf f}_b-\delta_{ik}{\bf e}_j\otimes {\bf f}_b\right)+\sum_{\substack{i<j;\,i,j\in I_p\\q\in I_p;\,a\in \{1,\dots,p\}}}\alpha_{ij}u_{qa}\left(\delta_{jq}{\bf e}_i\otimes {\bf f}_a-\delta_{iq}{\bf e}_j\otimes {\bf f}_a\right)+\\
+&\sum_{\substack{r\in I_p;\,s\notin I_p\\k\notin I_p;\,b\in \{1,\dots,p\}}}\beta_{rs}u_{kb}\left(\delta_{sk}{\bf e}_r\otimes {\bf f}_b-\delta_{rk}{\bf e}_s\otimes {\bf f}_b\right)+\sum_{\substack{r\in I_p;\,s\notin I_p\\q\in I_p;\,a\in \{1,\dots,p\}}}\beta_{rs}u_{qa}\left(\delta_{sq}{\bf e}_r\otimes {\bf f}_a-\delta_{rq}{\bf e}_s\otimes {\bf f}_a\right)\\
=&\sum_{\substack{i<j;\,i,j\in I_p\\a\in \{1,\dots,p\}}}\alpha_{ij}\left(u_{ja}{\bf e}_i\otimes {\bf f}_a-u_{ia}{\bf e}_j\otimes {\bf f}_a\right)+\sum_{\substack{r\in I_p;\,s\notin I_p\\a\in \{1,\dots,p\}}}\beta_{rs}\left(u_{sa}{\bf e}_r\otimes {\bf f}_a-u_{ra}{\bf e}_s\otimes {\bf f}_a\right).
\end{align*}
Decomposing the above matrix equality on the subspaces $\text{Span} \{{\bf e}_s\otimes {\bf f}_a\,|\,s\notin I_p\}$ and  $\text{Span} \{{\bf e}_l\otimes {\bf f}_a\,|\,l\in I_p\}$, we have
\begin{gather}
\sum_{\substack{i<j;\,i,j\in I_p\\a\in \{1,\dots,p\}}}\alpha_{ij}\left(u_{ja}{\bf e}_i\otimes {\bf f}_a-u_{ia}{\bf e}_j\otimes {\bf f}_a\right)+\sum_{\substack{r\in I_p;\,s\notin I_p\\a\in \{1,\dots,p\}}}\beta_{rs}u_{sa}{\bf e}_r\otimes {\bf f}_a=\mathbb{O}_{n\times p}\,;\label{alfabeta}\\
\sum_{\substack{s\notin I_p\\a\in \{1,\dots,p\}}}\left(\sum_{r\in I_p}\beta_{rs}u_{ra}\right){\bf e}_s\otimes {\bf f}_a=\mathbb{O}_{n\times p}.\label{beta}
\end{gather}

Considering now the matrix $[\boldsymbol{\beta}]\in \mathcal{M}_{p\times (n-p)}(\R)$, $[\boldsymbol{\beta}]:=\displaystyle\sum\limits_{a\in\{1,\dots,p\};\,s\notin I_p}\beta_{\tau^{-1}(a)s}{\bf f}_a\otimes {\bf h}_{\sigma(s)}$\footnote{We relabel the set $\{1,...,n\}\backslash I_p$ using  the unique strictly increasing function $\sigma: \{1,...,n\}\backslash I_p\rightarrow \{1,...,n-p\}$. Analogously, we relabel the set $I_p$ using the unique strictly increasing function $\tau: I_p\rightarrow \{1,...,p\}$. The vectors ${\bf h}_1,...,{\bf h}_{n-p}$ form the canonical basis of $\R^{n-p}$.
}, we can rewrite the equality \eqref{beta} in a condensed matrix form
$$[\boldsymbol{\beta}]^T\bar{U}=\mathbb{O}_{(n-p)\times p}.$$
Indeed, we have
\begin{align*}
[\boldsymbol{\beta}]^T\bar{U}=&\left(\sum_{\substack{s\notin I_p\\b\in \{1,\dots,p\}}}\beta_{\tau^{-1}(b)s} {\bf h}_{\sigma(s)}\otimes {\bf f}_b\right)\left(\sum\limits_
{\substack{
		r\in I_p \\
		a\in \{1,...,p\}
	}}
	u_{ra}\,{\bf f}_{\tau(r)}\otimes {\bf f}_a\right)\\
	=&\sum_{\substack{r\in I_p;\,s\notin I_p\\a,b\in \{1,\dots,p\}}}\beta_{\tau^{-1}(b)s}u_{ra}\delta_{b\tau(r)}{\bf h}_{\sigma(s)}\otimes {\bf f}_a\\
	=& \sum_{\substack{r\in I_p;\,s\notin I_p\\a\in \{1,...,p\}}}\beta_{rs}u_{ra}{\bf h}_{\sigma(s)}\otimes {\bf f}_a\\
	=&\sum_{\substack{s\notin I_p\\a\in \{1,...,p\}}}\left(\sum_{r\in I_p}\beta_{rs}u_{ra}\right){\bf h}_{\sigma(s)}\otimes {\bf f}_a\\
		=&\mathbb{O}_{(n-p)\times p}.
\end{align*}
Since the matrix $\bar{U}$ is invertible, we obtain that $[\boldsymbol{\beta}]=\mathbb{O}_{p\times (n-p)}$, which implies that
$\beta_{rs}=0$, for all $r\in I_p$ and all $s\notin I_p$.

Substituting these last equalities in \eqref{alfabeta}, it simplifies to
\begin{equation}\label{alfa}
\mathbb{O}_{n\times p}=\sum_{\substack{i<j;\,i,j\in I_p\\a\in \{1,\dots,p\}}}\alpha_{ij}\left(u_{ja}{\bf e}_i\otimes {\bf f}_a-u_{ia}{\bf e}_j\otimes {\bf f}_a\right).
\end{equation}
We introduce now the matrix $[\boldsymbol{\alpha}]\in \mathcal{M}_{p\times p}(\R)$,
$[\boldsymbol{\alpha}]:=\sum\limits_{i<j;\,i,j\in I_p}\alpha_{ij}\left({\bf f}_{\tau(i)}\otimes {\bf f}_{\tau(j)}-{\bf f}_{\tau(j)}\otimes {\bf f}_{\tau(i)}\right)$.
We have the following computations:
\begin{align*}
[\boldsymbol{\alpha}]\bar{U}=&\sum_{i<j;\,i,j\in I_p}\alpha_{ij}\left({\bf f}_{\tau(i)}\otimes {\bf f}_{\tau(j)}-{\bf f}_{\tau(j)}\otimes {\bf f}_{\tau(i)}\right)\left(\sum_{\substack{k\in I_p\\a\in \{1,\dots,p\}}}u_{ka}{\bf f}_{\tau(k)}\otimes {\bf f}_{a}\right)\\
=&\sum_{\substack{i<j;\,i,j\in I_p\\k\in I_p;\,a\in \{1,\dots,p\}}}\alpha_{ij}u_{ka}\left(\delta_{jk}{\bf f}_{\tau(i)}\otimes {\bf f}_{a}-\delta_{ik}{\bf f}_{\tau(j)}\otimes {\bf f}_{a}\right)\\
=&\sum_{\substack{i<j;\,i,j\in I_p \\a\in \{1,\dots,p\} }}\alpha_{ij}\left(u_{ja}{\bf f}_{\tau(i)}\otimes {\bf f}_{a}-u_{ia}{\bf f}_{\tau(j)}\otimes {\bf f}_{a}\right).
\end{align*}
By selecting from the matrix equality \eqref{alfa} the rows with indexes in $I_p$, we obtain that $[\boldsymbol{\alpha}]\bar{U}=\mathbb{O}_{p\times p}$, and since the matrix $\bar{U}$ is invertible it follows that $[\boldsymbol{\alpha}]=\mathbb{O}_{p\times p}$ and therefore $\alpha_{ij}=0$ for all $i,j\in I_p$ with $i<j$.

Thus, we have proved the linear independence of the vectors $\displaystyle\frac{\partial \varphi_U}{\partial \omega_{ij}}({\bf 0})$, that also form a basis for the tangent space $T_USt_p^n$.

\begin{prop}\label{Baza-2}
Let $U\in St_p^n$ and $1\leq i_1<...<i_p\leq n$ be the indexes of the rows that form a full rank submatrix $\bar{U}$ of $U$. Then the vectors: 
\begin{gather*}
\Gamma_{i'j'}(U):=\Lambda_{i'j'}U,\,\,i',j'\in I_p,\,\,i'<j', \\
\Gamma_{i''j''}(U):=\Lambda_{i''j''}U,\,\,\,i''\in I_p,\,j''\notin I_p,
\end{gather*}
form a basis for the tangent space $T_USt_p^n$.
\end{prop}
As a consequence, we have the following description for the tangent space to an orthogonal Stiefel manifold.
\begin{thm}\label{spatiutangent}
Let $U\in St^n_p$. Then
$$T_{U}St^n_p=\left\{\Omega U|~\Omega =\left[\omega_{ij}\right]\in \normalfont{\text{Skew}}_{n\times n}(\R),\normalfont{\text{where}}\,\, \omega_{ij}=0\,\,\normalfont{\text{for all}}\,\,i\notin I_p\,\normalfont{\text{and}}\, j\notin I_p\right\}.$$
\end{thm}

\section{Critical points of smooth functions defined on orthogonal Stiefel manifolds}

In this section, we give necessary and sufficient conditions for critical points of a smooth cost function defined on orthogonal Stiefel manifolds using the embedded vector field method introduced and used in \cite{Birtea-Comanescu-Hessian}, \cite{birtea-comanescu}, and \cite{birtea-comanescu-5-electron}. We apply these results to well-known cost functions as Procrustes and Penrose regression cost functions, sums of heterogeneous quadratic forms, and Brockett cost functions. We also discuss our results in comparison with previous results existing in the literature.

For a matrix $U\in  \mathcal{M}_{n\times p}(\R)$, we denote by ${\bf u}_1,...,{\bf u}_p\in \R^n$ the  vectors formed with the columns of the matrix $U$ and consequently, $U$ has the form $U=\left[{\bf u}_1,...,{\bf u}_p\right]$. If $U\in St_p^n=\{U\in \mathcal{M}_{n\times p}(\R) \,|\,U^TU=\mathbb{I}_p\}$, then the vectors ${\bf u}_1,...,{\bf u}_p\in \R^n$ are orthonormal.
We identify $\mathcal{M}_{n\times p}(\R)$ with $\R^{np}$ using the isomorphism $\text{vec}:\mathcal{M}_{n\times p}\rightarrow \R^{np}$ defined by $\text{vec}(U)\stackrel{\text{not}}{=}{\bf u}:=({\bf u}_1^T,...,{\bf u}_p^T)$.

The constraint functions $F_{aa},F_{bc}:\R^{np}\rightarrow \R$ that describe the Stiefel manifold as a preimage of a regular value are given by:
\begin{align}\label{constraints}
F_{aa} ({\bf u}) & =\frac{1}{2}||{\bf u}_a||^2,\,\,1 \leq a\leq p, \\
F_{bc} ({\bf u}) & = \left<{\bf u}_b,{\bf u}_c\right>,\,\,1\leq b<c\leq p. \label{constraints-2}
\end{align}
More precisely, we have ${\bf F}:\R^{np}\rightarrow \R^{\frac{p(p+1)}{2}}$, ${\bf F}:=\left( \dots , F_{aa},\dots ,F_{bc}, \dots\right)$, $$St_p^n\simeq {\bf F}^{-1}\left( \dots , \frac{1}{2},\dots ,0, \dots\right)\subset \R^{np}.$$
Consider a smooth cost function $\widetilde{G}:St^n_p\rightarrow \R$.  
In what follows we will address the problem of finding the critical points of the cost function $\widetilde{G}$ defined on the Stiefel manifold. In order to solve this problem, we consider a smooth extension $G:\R^{np}\rightarrow \R$ of the cost function $\widetilde{G}=G_{|_{St^n_p}}$ and we use the embedded gradient vector field method presented in \cite{Birtea-Comanescu-Hessian}, \cite{birtea-comanescu}, and \cite{birtea-comanescu-5-electron}.
The embedded gradient vector field is defined on the open set formed with the regular leaves of the constraint function and it  has the formula:
$$\partial G({\bf u})=\nabla G({\bf u})-\sum_{1\leq a\leq p}\sigma_{aa}({\bf u})\nabla F_{aa}({\bf u})-\sum_{1\leq b< c\leq p}\sigma_{bc}({\bf u})\nabla F_{bc}({\bf u}),$$
where $\sigma_{aa},\sigma_{bc}$ are the Lagrange multiplier functions. 

Using the property $(\partial G)_{|St^n_p}=\nabla_{{\bf g}_{\text{ind}}^{St^n_p}}\widetilde{G}$ proved in \cite{Birtea-Comanescu-Hessian} and \cite{birtea-comanescu}, we have the following necessary and sufficient conditions for a critical point of the cost function $\widetilde{G}$.

\begin{thm}
An element $U\in St^n_p$ is a critical point of the cost function $\widetilde{G}$ if and only if $\partial G({\bf u})={\bf 0}$.
\end{thm}

In the case of orthogonal constraints the Lagrange multiplier functions, see \cite{Birtea-Comanescu-Hessian}, are given by the formulas: 
\begin{equation}\label{Lagrange-multipliers-functions}
\left.\begin{array}{l}
\sigma_{aa}({\bf u})=  \left<\nabla G({\bf u}),\nabla F_{aa}({\bf u})\right>=\left<\displaystyle\frac{\partial G}{\partial {\bf u}_a}({\bf u}),{\bf u}_a\right>;\\ \\
\sigma_{bc}({\bf u})= \left<\nabla G({\bf u}),\nabla F_{bc}({\bf u})\right>=\displaystyle\frac{1}{2}\left(\left<\displaystyle\frac{\partial G}{\partial {\bf u}_c}({\bf u}),{\bf u}_b\right>+\left<\displaystyle\frac{\partial G}{\partial {\bf u}_b}({\bf u}),{\bf u}_c\right>\right).\end{array}\right.
\end{equation}
Note that in general $\left<\displaystyle\frac{\partial G}{\partial {\bf u}_c}({\bf u}),{\bf u}_b\right>\neq \left<\displaystyle\frac{\partial G}{\partial {\bf u}_b}({\bf u}),{\bf u}_c\right>$. 

If $U$ is a critical point of $\widetilde{G}$, then $\sigma_{aa}({\bf u})$, $\sigma_{bc}({\bf u})$ become the classical Lagrange multipliers. 
The embedded gradient vector field is a more explicit form of the equivalent projected gradient vector field described in \cite{rosen}.
The solutions of the equation $\partial G({\bf u})={\bf 0}$ are critical points for the function $G$ restricted to regular leaves of the constraint functions. 
Consequently, using again the identification $\text{vec}(U)={\bf u}$,  a matrix $U\in \mathcal{M}_{n\times p}(\R)$ is a critical point for the cost function $\widetilde{G}=G_{|_{St^n_p}}$ if and only if $\partial G({\bf u})={\bf 0}$ and $U^TU=\mathbb{I}_p$, or equivalently, 
\begin{equation}\label{ecuatie-embedded}
\begin{cases}
\nabla G({\bf u})-\sum\limits_{1\leq a\leq p}\sigma_{aa}({\bf u})\nabla F_{aa}({\bf u})-\sum\limits_{1\leq b< c\leq p}\sigma_{bc}({\bf u})\nabla F_{bc}({\bf u})={\bf 0} \\
U^TU=\mathbb{I}_p.
\end{cases}
\end{equation}

Next, we give necessary and sufficient conditions  for a critical point.

\begin{thm}\label{egregium}
	A matrix $U\in \mathcal{M}_{n\times p}(\R)$ is a critical point for the cost function $\widetilde{G}=G_{|_{St^n_p}}$ {\bf if and only if} the following conditions are simultaneously satisfied:
	\begin{itemize}
		
		\item [(i)] $\left<\displaystyle\frac{\partial G}{\partial {\bf u}_c}({\bf u}),{\bf u}_b\right>=\left<\displaystyle\frac{\partial G}{\partial {\bf u}_b}({\bf u}),{\bf u}_c\right>,\,\,\,\text{for all}\,\,\,1\leq b<c\leq p$;
		
		\item [(ii)] $ \displaystyle\frac{\partial G}{\partial {\bf u}_a}({\bf u})\in {\normalfont \text{Span}}\{{\bf u}_1,...,{\bf u}_p\}
,\,\,\,\text{for all}\,\,1\leq a\leq p$;
		
		\item [(iii)] $U^TU=\mathbb{I}_p$,
	\end{itemize}
	where ${\bf u}=\normalfont{\text{vec}}(U)$.
\end{thm}

\begin{proof}
First, we will prove that the conditions $(i)$, $(ii)$ and $(iii)$ of the Theorem are necessary. By straightforward computations, the hypothesis $\partial G({\bf u})={\bf 0}$ is equivalent with the following system of equations:
\begin{equation}\label{embedded-explicit}
\begin{cases}
\displaystyle\frac{\partial G}{\partial {\bf u}_a}({\bf u})-\sum\limits_{d=1}^a\sigma_{da}({\bf u}){\bf u}_d-\sum\limits_{d=a+1}^p\sigma_{ad}({\bf u}){\bf u}_d={\bf 0},\,\,\,\forall a\in \{1,...,p-1\} \\
\displaystyle\frac{\partial G}{\partial {\bf u}_p}({\bf u})-\sum\limits_{d=1}^p\sigma_{dp}({\bf u}){\bf u}_d={\bf 0},
\end{cases}
\end{equation}	
where the Lagrange multiplier functions $\sigma_{ad}$ are given by \eqref{Lagrange-multipliers-functions}. 
From the above equalities it follows $(ii)$. From $(ii)$ and the hypothesis that $U\in St^n_p$ we have:
\begin{equation}\label{orto}
\frac{\partial G}{\partial {\bf u}_a}({\bf u})=\sum_{d=1}^p\left<\frac{\partial G}{\partial {\bf u}_a}({\bf u}),{\bf u}_d\right>{\bf u}_d.
\end{equation}
Substituting \eqref{orto} into \eqref{embedded-explicit} and using the linear independence of the vectors ${\bf u}_1,...,{\bf u}_p$ formed with the columns of the matrix $U$, we obtain $(i)$.

For sufficiency, we solve the set of equations $(i)$ and $(ii)$. This implies that
\begin{equation*}
\frac{\partial G}{\partial {\bf u}_a}({\bf u})=\sum_{d=1}^p\lambda_d{\bf u}_d.
\end{equation*} 
Among these solutions, we choose the ones that belong to $St^n_p$, and consequently for those solutions we obtain
$$\lambda_d=\left<\frac{\partial G}{\partial {\bf u}_a}({\bf u}),{\bf u}_d\right>.$$
Using $(i)$ and the fact that
\begin{equation*}
\frac{\partial G}{\partial {\bf u}_a}({\bf u})=\sum_{d=1}^p\left<\frac{\partial G}{\partial {\bf u}_a}({\bf u}),{\bf u}_d\right>{\bf u}_d.
\end{equation*} 
we obtain the desired equality $\partial G({\bf u})={\bf 0}$. 
\end{proof}

If the condition $(iii)$ of the above theorem is satisfied, then the condition $(ii)$  can be replaced with
\begin{equation}
(\mathbb{I}_n-UU^T)\frac{\partial G}{\partial {\bf u}_a}({\bf u})={\bf 0},\,\,\,\text{for all}\,\,1\leq a\leq p,
\end{equation}
where the matrix $\mathbb{I}_n-UU^T$ is associated to the orthogonal projection in $\R^n$ onto the subspace normal to $\text{Span}\{{\bf u}_1,...,{\bf u}_p\}$. The necessity of conditions $(i)$ and $(ii)$ have been  previously discovered in \cite{chu} in the context of orthogonal Procrustes problem and Penrose regression problem. We will present the details later in the paper.

The above Theorem shows that, in order to find the critical points of the cost function $\widetilde{G}=G_{|_{St^n_p}}$ it is necessary and sufficient to solve the system of equations $(i)$ and $(ii)$ and among them choose the ones that belong to the orthogonal Stiefel manifold $St^n_p$.

The necessary and sufficient conditions of the above theorem are natural for orthogonal Stiefel manifolds in the sense that orthogonal Stiefel manifolds are in-between the sphere ($p=1$) and the orthogonal group ($p=n$).
In the case when $p=1$, we obtain that the necessary and sufficient conditions of Theorem \ref{egregium} reduce to the radial condition $\displaystyle\frac{\partial G}{\partial {\bf u}}({\bf u})=\lambda{\bf u}$, $\lambda\in \R$ for a critical point of a function restricted to a sphere. When $p=n$, we are in the case of orthogonal group $SO(n)$ and  the necessary and sufficient conditions of Theorem \ref{egregium} reduce to the symmetric condition $(i)$.

In order to formulate the necessary and sufficient conditions of Theorem \ref{egregium} in a matrix form, we need to write the  embedded vector field $\partial G$ in a matrix form.
For the following considerations we make the notations 
$$\nabla G(U):=\text{vec}^{-1}(\nabla G({\bf u}))\in \mathcal{M}_{n\times p}(\R),$$ 
$$\partial G(U):=\text{vec}^{-1}(\partial G({\bf u}))\in \mathcal{M}_{n\times p}(\R).$$
We introduce the symmetric matrix
$$\Sigma(U): =\left[\sigma_{bc}({\bf u})\right]\in \mathcal{M}_{p\times p}(\R),$$
where we define $\sigma_{cb}({\bf u}):=\sigma_{bc}({\bf u})$ for $1\leq b<c\leq p$.
 For the particular case of orthogonal Stiefel manifold, by a straightforward computation using \eqref{Lagrange-multipliers-functions}, we have 
\begin{equation}
\Sigma(U)=\frac{1}{2}\left(\nabla G(U)^TU+U^T\nabla G(U)\right).
\end{equation}

Computing $\text{vec}^{-1}\left(\nabla F_{aa}({\bf u})\right)$ and $\text{vec}^{-1}(\nabla F_{bc}({\bf u}))$, the matrix form of the embedded gradient vector field is given by 
\begin{equation}
\partial G(U)=\nabla G(U)-U\Sigma(U).
\end{equation}
From the geometrical point of view, the vector $\nabla G(U)\in \mathcal{M}_{n\times p}(\R)$ does not in general belong to the tangent space $T_USt^n_p$. The vector $-U\Sigma(U)$ is a correcting term so that the vector $\partial G(U)\in T_USt^n_p$ for every $U\in St^n_p$. This has been proved in  \cite{Birtea-Comanescu-Hessian}.
The matrix form of the system of equations \eqref{ecuatie-embedded} is given by
\begin{equation}\label{ecuatie-matriciala-embedded}
\begin{cases}
\nabla G(U)-U\Sigma(U)={\bf 0} \\
U^TU=\mathbb{I}_p.
\end{cases}
\end{equation}

The conditions of Theorem \ref{egregium} can be written in matrix form in the following way. 

\begin{thm}\label{second-egregium}
A matrix $U\in \mathcal{M}_{n\times p}(\R)$ is a critical point for the cost function $\widetilde{G}=G_{|_{St^n_p}}$ { if and only if} the following conditions are simultaneously satisfied:
\begin{equation}\label{conditii-matriceale}
\left.
\begin{array}{l}
(i)\,\,U^T\nabla G(U)=\nabla G(U)^TU; \\
 \\
(ii)\,\,\nabla G(U)=UU^T\nabla G(U); \\
\\
(iii)\,\,U^TU=\mathbb{I}_p.
\end{array}\right.
\end{equation}

	%
%
%
\end{thm}

Using the classical Lagrange multiplier approach for constraint optimization problems, in \cite{wen} the equations that have to be solved in order to find the critical points of the cost function $\widetilde{G}=G_{|_{St^n_p}}$ are 
\begin{equation}\label{eq-critic}
\begin{cases}
\nabla G(U)-U\nabla G(U)^TU=\mathbb{O}_{n\times p} \\
U^TU=\mathbb{I}_p,
\end{cases}
\end{equation}
which is an equivalent matrix form for the system of equations \eqref{conditii-matriceale} in the case of orthogonal Stiefel manifold. {\bf Note that the vector field $\nabla G(U)-U\nabla G(U)^TU\neq \partial G(U)$ when $U\in ST^n_p$ is not a critical point of $G$.}

Also, in the same paper \cite{wen} the following equivalent necessary and sufficient conditions for critical points have been obtained:
\begin{equation}\label{conditie-chinezi}
\begin{cases}
\nabla G(U)U^T-U\nabla G(U)^T=\mathbb{O}_{n\times n} \\
U^TU=\mathbb{I}_p.
\end{cases}
\end{equation}
We give a short proof of the equivalence between the equations \eqref{eq-critic} and \eqref{conditie-chinezi}. Multiplying \eqref{conditie-chinezi} to the right with the matrix $U$ and using the Stiefel condition $U^TU=\mathbb{I}_p$, we obtain \eqref{eq-critic}. Now assume that \eqref{eq-critic} holds, i.e. $\nabla G(U)=U\nabla G(U)^TU=U\left(U^T\nabla G(U)U^T\right)U=UU^T\nabla G(U)$. Multiplying to the right with $U^T$ we obtain
$\nabla G(U)U^T=U\left(U^T\nabla G(U)U^T\right)=U\nabla G(U)^T$, which is exactly the first equation of \eqref{conditie-chinezi}.

Consequently, necessary and sufficient conditions of Theorem \ref{second-egregium} are equivalent with necessary and sufficient conditions \eqref{conditie-chinezi} obtained in \cite{wen}. The difference between the two sets of necessary and sufficient conditions is that in Theorem \ref{second-egregium} the equations imply natural relations involving the columns of the matrix $U$, i.e. components of the orthonormal vectors ${\bf u}_1,...,{\bf u}_p$, while \eqref{conditie-chinezi} involves equations containing the lines of the matrix $U$. 

Using the particularity of the Stiefel constraints, in \cite{rapcsak} are given other equivalent conditions with those from \eqref{conditii-matriceale}.
\medskip

{\bf Critical points for orthonormal Procrustes cost function}. We consider the following  optimization problem:
\begin{equation}
\left.
\begin{array}{l}
\hbox{Minimize}\, ||AU-B||^2 \\
U^TU=\mathbb{I}_p
\end{array}\right.,
\end{equation}
where $A\in \mathcal{M}_{m\times n}(\R)$, $B\in \mathcal{M}_{m\times p}(\R)$, $U\in \mathcal{M}_{n\times p}(\R)$, and $||\cdot ||$ is the Frobenius norm. The cost function associated to this optimization problem is given by $\widetilde{G}:St^n_p\rightarrow \R$ and its natural extension $G:\R^{np}\rightarrow \R$, 
$$G({\bf u})=\frac{1}{2}||AU-B||^2=\frac{1}{2}\hbox{tr}(U^TA^TAU)-\hbox{tr}(U^TA^TB)+\frac{1}{2}\hbox{tr}(B^TB).$$
In the following, we will give the specifics of the necessary and sufficient conditions of Theorem \ref{second-egregium} in the case of Procrustes cost function.
By a straightforward computation we have that
$$\nabla G(U)=A^TAU-A^TB.$$ 
Consequently, the condition $(i)$ is equivalent with the symmetry of the matrix $U^TA^TAU-B^TAU$. As the matrix $U^TA^TAU$ is symmetric, we obtain that condition $(i)$ of Theorem \ref{second-egregium} is equivalent with the symmetry of the matrix $B^TAU$. This condition has been previously obtained in \cite{chu}. The condition $(ii)$ of the Theorem \ref{second-egregium} is equivalent with $(\mathbb{I}_n-UU^T)(A^TAU-A^TB)={\bf 0}$, condition also previously obtained in \cite{chu}. The following result shows that the necessary conditions presented in \cite{chu} for the Procrustes cost function are also sufficient conditions.

\begin{thm}
A matrix $U\in St^n_p$ is a critical point for the Procrustes cost function if and only if:
\begin{itemize}
\item [(i)] the matrix $B^TAU$ is symmetric;
\item [(ii)] $(\mathbb{I}_n-UU^T)(A^TAU-A^TB)={\bf 0}$.

\end{itemize}

\end{thm}
A different approach for studying critical points of Procrustes problem using normal and secular equations has been undertaken in \cite{elden}.
\medskip

{\bf Critical points for Penrose regression cost function.} The Penrose regression problem is the following optimization problem:
\begin{equation}
\left.
\begin{array}{l}
\hbox{Minimize}\, ||AUC-B||^2 \\
U^TU=\mathbb{I}_p
\end{array}\right.,
\end{equation}
where $A\in \mathcal{M}_{m\times n}(\R)$, $B\in \mathcal{M}_{m\times q}(\R)$, $C\in \mathcal{M}_{p\times q}(\R)$, $U\in \mathcal{M}_{n\times p}(\R)$, and $||\cdot ||$ is the Frobenius norm.  The cost function associated to this optimization problem is given by $\widetilde{G}:St^n_p\rightarrow \R$ and its natural extension $G:\R^{np}\rightarrow \R$, 
$$G({\bf u})=\frac{1}{2}||AUC-B||^2=\frac{1}{2}\hbox{tr}(C^TU^TA^TAUC)-\hbox{tr}(C^TU^TA^TB)+\frac{1}{2}\hbox{tr}(B^TB).$$
By a straightforward computation, we have that
$$\nabla G(U)=A^T(AUC-B)C^T.$$ 
The necessary and sufficient conditions of Theorem \ref{second-egregium} for critical points become in this case:

\begin{thm}
A matrix $U\in St^n_p$ is a critical point for the Penrose regression cost function if and only if:
\begin{itemize}
\item [(i)] the matrix $C(AUC-B)^TAU$ is symmetric;
\item [(ii)] $(\mathbb{I}_n-UU^T)A^T(AUC-B)C^T={\bf 0} .$
\end{itemize}
\end{thm}
These conditions have been previously found in \cite{chu} as necessary conditions for critical points of the Penrose regression cost function.
\medskip

{\bf Critical points for sums of heterogeneous quadratic forms.} 
Consider the following optimization problem on orthogonal Stiefel manifold $St^n_p$, extensively studied in \cite{balogh} and \cite{rapcsak}: 
\begin{equation}
\left.
\begin{array}{l}
\hbox{Minimize}\, \sum\limits_{i=1}^p{\bf u}_i^TA_i{\bf u}_i \\
U^TU=\mathbb{I}_p
\end{array}\right.,
\end{equation}
where $A_i$ are $n\times n$ symmetric matrices and ${\bf u}_i$ are the column vectors of the the matrix $U\in St^n_p$. 
By a straightforward computation, we have that
$$\nabla G(U)=\left[A_1{\bf u}_1,...,A_p{\bf u}_p\right].$$ 
The necessary and sufficient conditions of Theorem \ref{second-egregium} for critical points become in this case:
\begin{thm}
A matrix $U\in St^n_p$ is a critical point for the cost function $\sum\limits_{i=1}^p{\bf u}_i^TA_i{\bf u}_i$ if and only if:
\begin{itemize}
\item [(i)] $U^TA(U)=A(U)^TU$;
\item [(ii)] $A(U)=UU^TA(U),$
\end{itemize}
where we have made the notation $A(U)=\left[A_1{\bf u}_1,...,A_p{\bf u}_p\right]$.
\end{thm}

The above necessary and sufficient conditions are the same conditions discovered in \cite{bolla} (eq. (3.3) and (3.4) from the proof of the Theorem 3.1).

A particular case of the cost function $\sum\limits_{i=1}^p{\bf u}_i^TA_i{\bf u}_i$ is when $A_i=\mu_i A$, where $0\leq \mu_1\leq ...\leq \mu_p$ and $A$ is a $n\times n$ symmetric matrix. Thus, we obtain the Brockett cost function
$$G_B=\sum\limits_{i=1}^p\mu_i{\bf u}_i^TA{\bf u}_i.$$
For this cost function the conditions $(i)$ and $(ii)$ of Theorem \ref{egregium} become:
\begin{itemize}
\item [(i)] $(\mu_b-\mu_c){\bf u}_b^TA{\bf u}_c=0,\,\,\forall b,c\in \{1,...,p\},\,\,\,b\neq c$;
\item [(ii)] $\mu_aA{\bf u}_a\in \hbox{Span}\{{\bf u}_1,...,{\bf u}_p\},\,\,\forall a\in \{1,...,p\}$.

\end{itemize}
Depending on the parameters $\mu_1,...,\mu_p$ the above two conditions can be further explained.

I.  If the parameters $\mu_1,...,\mu_p$ are pairwise distinct and strictly positive, then $U\in St^n_p$ is a critical point of the Brockett cost function if and only if every column vector of the matrix $U$ is an eigenvector of the matrix $A$.

II. For the case when among the strictly positive parameters $\mu_1,...,\mu_p$ we have multiplicity, the set of critical points becomes larger. More precisely, we have $$0<\mu_1=...=\mu_{s_1}<\mu_{s_1+1}=...=\mu_{s_1+s_2}<\dots<\mu_{s_1+...+s_{q-1}+1}=...=\mu_{s_1+...+s_{q}},$$ where $s_1\geq 1,...,s_q\geq 1$ and $s_1+...+s_q=p$. We denote by $J_1=\{1,...,s_1\}$, ...,$J_q=\{s_1+s_2+...+s_{q-1}+1,...,p\}$. By an elementary computation, the matrix $U\in St^n_p$ is a critical point of the Brockett cost function if and only if $\text{Span}\{{\bf u}_k\,|\,k\in J_{l}\}$ is an invariant subspace of $A$ for all $l\in \{1,...,q\}$.

We illustrate the above results on a simple case of a Brockett cost function defined on $St_2^3$. Assume that the matrix $A$ is in diagonal form with distinct entries. If $0<\mu_1<\mu_2$, then we are in the case I and $U=[{\bf u}_1,{\bf u}_2]$ is a critical point for the Brockett cost function if and only if ${\bf u}_1=\pm {\bf e}_i$ and ${\bf u}_2=\pm {\bf e}_j$ with $i,j\in \{1,2,3\}$ and $i\neq j$. If $0<\mu_1=\mu_2$, then we are in the case II and  $U=[{\bf u}_1,{\bf u}_2]$ is a critical point for the Brockett cost function if and only if the set $\{{\bf u}_1,{\bf u}_2\}$ is an orthonormal frame of any coordinate plane $\text{Span}\{{\bf e}_i,{\bf e}_j\}$ with $i\neq j$.

\section{Steepest descent algorithm}

Let $(S,{\bf g}_{_S})$ be a smooth Riemannian manifold and $\wt{G}:S\rightarrow \R$ be a smooth cost function. The iterative scheme of steepest descent  is given by
\begin{equation}
x_{k+1}={\mathcal{R}}_{x_k}(-\lambda_k\nabla_{{\bf g}_{_S}}\wt{G}(x_k)),
\end{equation}
where ${\mathcal{R}}:TS\rightarrow S$ is a smooth retraction, notion introduced in \cite{shub} (see also \cite{adler}), and $\lambda_k\in\R$ is a scalar called step length. For the case when the manifold $S$ is the preimage of a regular value of a set of constraint functions we have that $\nabla_{{\bf g}_{_S}}\wt{G}(x_k)=\partial G(x_k)$, where $G$ is an extension to the ambient space $M$ of the cost function $\wt{G}$, $\partial G$ is the embedded gradient vector field introduced in \cite{Birtea-Comanescu-Hessian}, and ${\bf g}_S$ is the induced Riemannian metric on $S$ by the ambient space $M$.  The vector $\partial G(x_k)$ is written in the coordinates of the ambient space $M$, but it belongs to the tangent space $T_{x_k}S$ viewed as a subspace of $T_{x_k}M$.

If the manifold $S$ is locally diffeomorphic with a manifold $N$ via a local diffeomorphism $f$ and we know a retraction $\mathcal{R}^N$ for the manifold $N$, then 
$$\mathcal{R}^{f}=f\circ \mathcal{R}^N\circ \left(Df\right)^{-1}$$
is a retraction for the manifold $S$.
A particular case of the above construction is when we replace the local diffeomorphism with local charts. For the case of an orthogonal Stiefel manifold we will use the local charts $\varphi_U$ defined by \eqref{charts}. More precisely, the retraction induced by a local chart $\varphi_U$ is given by
\begin{equation}\label{retracta}
\mathcal {R}^{\varphi_{_U}}(\Omega U)=\mathcal{C}\left(\frac{1}{2}\Omega\right)U,
\end{equation}
where $\Omega\in W_{I_p}$.
\medskip

{\bf Steepest descent algorithm on orthogonal Stiefel manifolds}:
\begin{framed}
\begin{itemize}

\item [1.] For a $n\times p$ matrix $U$ construct the vector ${\bf u}:=\text{vec}(U)=\left({\bf u}_1^T, ...,{\bf u}_p^T\right)\in \R^{np}$.

\item [2.] Consider a smooth prolongation $G:\R^{np}\rightarrow \R$ of the cost function  $\wt{G}:St^n_p\rightarrow \R$.

\item [3.] Compute $\nabla G({\bf u})$ and construct the $n\times p$ matrix $\nabla G(U)=\text{vec}^{-1}(\nabla G({\bf u}))$.

\item [4.] Compute the Lagrange multiplier functions 
\begin{equation*}
\left.\begin{array}{l}
\sigma_{aa}({\bf u})=  \left<\displaystyle\frac{\partial G}{\partial {\bf u}_a}({\bf u}),{\bf u}_a\right>,\,\,
\sigma_{bc}({\bf u})=\displaystyle\frac{1}{2}\left(\left<\displaystyle\frac{\partial G}{\partial {\bf u}_c}({\bf u}),{\bf u}_b\right>+\left<\displaystyle\frac{\partial G}{\partial {\bf u}_b}({\bf u}),{\bf u}_c\right>\right).\end{array}\right.
\end{equation*}

\item [5.] Construct the symmetric $p\times p$ matrix $\Sigma(U)=\left[\sigma_{bc}({\bf u})\right]$.

\item [6.]  Compute the $n\times p$ matrix $\partial G(U)=\nabla G(U)-U\Sigma(U)$. 

\item [7.] Input ${U}_0\in St^n_p$ and $k=0$.

\item [8.] {\bf repeat}

$\bullet$ Compute the $n\times p$ matrix $\partial G(U_k)$.

$\bullet$ Determine a set $I_p(U_k)$ containing the indexes of  the rows that form a full rank submatrix of $U_k$.

$\bullet$ Construct a generic $n\times n$ skew-symmetric matrix $\Omega_k=[\omega_{ij}]$ in 
$$W_{I_p(U_k)}=\left\{\Omega=\left[\omega_{ij}\right]\in \text{Skew}_{n\times n}(\R)\,\left|\, \omega_{ij}=0\,\,\text{for all}\,\,i\notin I_p(U_k),\, j\notin I_p(U_k) \right.\right\}.$$

$\bullet$ Choose a length step $\lambda_k\in \R$ and solve the matrix equation of $np-\frac{p(p+1)}{2}$ variables $\omega_{ij}$
$$-\lambda_k\partial G(U_k)=\Omega_kU_k.$$

$\bullet$ Using the solution $\Omega_k$ of the above equation, compute $$U_{k+1}=\left(\mathbb{I}_n+\frac{1}{2}\Omega_k\right)\left(\mathbb{I}_n-\frac{1}{2}\Omega_k\right)^{-1}U_k.$$

{\bf until} $U_{k+1}$ sufficiently minimizes $\wt{G}$.

\end{itemize}
\end{framed}

The matrix equation
\begin{equation}\label{sistem-omega}
-\lambda_k\partial G(U_k)=\Omega_kU_k
\end{equation}
has a unique solution since $-\lambda_k\partial G(U_k)\in T_{U_k}St^n_p$  and this tangent vector is uniquely written as $-\lambda_k\partial G(U_k)=\omega_{i'j'}\Lambda_{i'j'}U_k+\omega_{i''j''}\Lambda_{i''j''}U_k$, see Proposition \ref{Baza-2}.

\medskip 

Next we will describe a method for finding the explicit solution of the equation \eqref{sistem-omega}.
Once we have computed $U_k$, we choose a set of indexes $I_p(U_k)=\{i_1,...,i_p\}$ that give a full rank submatrix of $U_k$ (the set of indexes $I_p(U_k)$ is not in general unique). We consider a permutation $\nu_k:\{1,...,n\}\rightarrow \{1,...,n\}$ such that $\nu_k(i_1)=1,...,\nu_k(i_p)=p$
and we introduce the permutation matrix 
\begin{equation}
\label{permutare}
P_{\nu_k^{-1}}=\left[\begin{array}{c}
{\bf e}_{\nu_k^{-1}(1)} \\
\vdots \\
{\bf e}_{\nu_k^{-1}(n)}
\end{array}
\right]\in\mathcal{M}_{n\times n}(\R).
\end{equation}
We make the following notations 
$$\widetilde{U}_k:=P_{\nu_k^{-1}}\cdot U_k\,\,\,\text{and}\,\,\,
\widetilde{\partial G}(U_k):=P_{\nu_k^{-1}}\cdot \partial G(U_k).$$
The matrix $\widetilde{U}_k\in St^n_p$ and it has the form
$$\widetilde{U}_k=\left[
\begin{array}{c}
	\widebar{U}_k \\
	\widebar{\widebar{U}}_k
\end{array}
\right],$$
where $\bar{U}_k \in\mathcal{M}_{p\times p}(\R)$ is an invertible matrix. According to Theorem \ref{spatiutangent}, we have that the tangent vectors in $T_{\widetilde{U}_k}St^n_p$ are of the form $\widetilde{\Omega}_k\widetilde{U}_k$, where
$$\widetilde{\Omega}_k=\left[
\begin{array}{cc}
\widebar{\Omega}_k & \widebar{\widebar{\Omega}}_k \\
-\widebar{\widebar{\Omega}}_k^T & \mathbb{O}
\end{array}
\right].$$
The equation $-\lambda_k \widetilde{\partial G}(U_k)=\widetilde{\Omega}_k\widetilde{U}_k$ has the equivalent form
\begin{equation*}
-\lambda_k
\left[
\begin{array}{c}
\widebar{\partial G}(U_k) \\
\widebar{\widebar{\partial G}}(U_k)
\end{array}
\right]=
\left[
\begin{array}{cc}
\widebar{\Omega}_k & \widebar{\widebar{\Omega}}_k \\
-\widebar{\widebar{\Omega}}_k^T & \mathbb{O}
\end{array}
\right]
\left[
\begin{array}{c}
\widebar{U}_k \\
\widebar{\widebar{U}}_k
\end{array}
\right],
\end{equation*}
where we have denoted
$$\widetilde{\partial G}(U_k):=\left[
\begin{array}{c}
	\widebar{\partial G}(U_k) \\
	\widebar{\widebar{\partial G}}(U_k)
\end{array}
\right].$$
By a straightforward computation, the above system has the solution
\begin{equation}
\begin{cases}
\widebar{\Omega}_k=-\lambda_k\left(\widebar{\partial G}(U_k)+\widebar{U}_k^{-T}\widebar{\widebar{\partial G}}(U_k)^T\widebar{\widebar{U}}_k\right)\widebar{U}_k^{-1} \\
\widebar{\widebar{\Omega}}_k=\lambda_k\widebar{U}_k^{-T}\widebar{\widebar{\partial G}}(U_k)^T.
\end{cases}
\end{equation}
Next, we will prove that the skew-symmetric matrix
$$\Omega_k:=P_{\nu_k^{-1}}^T\cdot \widetilde{\Omega}_k\cdot
P_{\nu_k^{-1}}$$
is the unique solution of equation \eqref{sistem-omega}. Indeed, 
\begin{align*}
\Omega_kU_k&=\left(P_{\nu_k^{-1}}^T\widetilde{\Omega}_kP_{\nu_k^{-1}}\right)P_{\nu_k^{-1}}^T\widetilde{U}_k=P_{\nu_k^{-1}}^T\widetilde{\Omega}_k\widetilde{U}_k\\
&=P_{\nu_k^{-1}}^T\left(-\lambda_k \widetilde{\partial G}(U_k)\right)=-\lambda_kP_{\nu_k^{-1}}^TP_{\nu_k^{-1}}\cdot \partial G(U_k)\\
&=-\lambda_k\partial G(U_k),
\end{align*}
where we have used the property that the permutation matrices are invertible and their inverse is equal with the transpose matrix.

According to \eqref{retracta}, we obtain after $k$ iterations 
\begin{align*}
U_{k+1}&=\left(\mathbb{I}_n+\frac{1}{2}{\Omega}_k\right)\left(\mathbb{I}_n-\frac{1}{2}{\Omega}_k\right)^{-1}{U}_k\\
&=\left(\mathbb{I}_n+\frac{1}{2}P_{\nu_k^{-1}}^T\wt{\Omega}_kP_{\nu_k^{-1}}\right)\left(\mathbb{I}_n-\frac{1}{2}P_{\nu_k^{-1}}^T\wt{\Omega}_kP_{\nu_k^{-1}}\right)^{-1}P_{\nu_k^{-1}}^T\wt{U}_k\\
&=P_{\nu_k^{-1}}^T\left(\mathbb{I}_n+\frac{1}{2}\wt{\Omega}_k\right)\left(\mathbb{I}_n-\frac{1}{2}\wt{\Omega}_k\right)^{-1}\wt{U}_k.
\end{align*}

{\bf The following is an alternative box that describes the steepest descent algorithm on orthogonal Stiefel manifolds:}
\begin{framed}
	\begin{itemize}
		
		\item [1.] For a $n\times p$ matrix $U$ construct the vector ${\bf u}:=\text{vec}(U)=\left({\bf u}_1^T, ...,{\bf u}_p^T\right)\in \R^{np}$.
		
		\item [2.] Consider a smooth prolongation $G:\R^{np}\rightarrow \R$ of the cost function  $\wt{G}:St^n_p\rightarrow \R$.
		
		\item [3.] Compute $\nabla G({\bf u})$ and construct the $n\times p$ matrix $\nabla G(U)=\text{vec}^{-1}(\nabla G({\bf u}))$.
		
		\item [4.] Compute the Lagrange multiplier functions 
		\begin{equation*}
		\left.\begin{array}{l}
		\sigma_{aa}({\bf u})=  \left<\displaystyle\frac{\partial G}{\partial {\bf u}_a}({\bf u}),{\bf u}_a\right>,\,\,
		\sigma_{bc}({\bf u})=\displaystyle\frac{1}{2}\left(\left<\displaystyle\frac{\partial G}{\partial {\bf u}_c}({\bf u}),{\bf u}_b\right>+\left<\displaystyle\frac{\partial G}{\partial {\bf u}_b}({\bf u}),{\bf u}_c\right>\right).\end{array}\right.
		\end{equation*}
		
		\item [5.] Construct the symmetric $p\times p$ matrix $\Sigma(U)=\left[\sigma_{bc}({\bf u})\right]$.
		
		\item [6.]  Compute the $n\times p$ matrix $\partial G(U)=\nabla G(U)-U\Sigma(U)$. 
		
		\item [7.] Input ${U}_0\in St^n_p$ and $k=0$.
		
		\item [8.] {\bf repeat}
		
		$\bullet$ Compute the $n\times p$ matrix $\partial G(U_k)$.
		
		$\bullet$ Determine a set $I_p(U_k)$ containing the indexes of the rows that form a full rank submatrix of $U_k$. Construct the permutation matrix $P_{\nu_k^{-1}}$ using formula \eqref{permutare}.
		
		$\bullet$ Compute $\widetilde{U}_k=P_{\nu_k^{-1}}\cdot U_k$ and
		$\widetilde{\partial G}(U_k)=P_{\nu_k^{-1}}\cdot \partial G(U_k)$. Write $\widetilde{U}_k$ and $\widetilde{\partial G}(U_k)$ in the block matrix form 
		$\left[
		\begin{array}{c}
			\widebar{U}_k \\
			\widebar{\widebar{U}}_k
		\end{array}
		\right]$, and respectively $\left[
		\begin{array}{c}
			\widebar{\partial G}(U_k) \\
			\widebar{\widebar{\partial G}}(U_k)
		\end{array}
		\right]$.
		
		$\bullet$ Choose a length step $\lambda_k\in \R$ and compute \begin{equation*}
		\begin{cases}
		\widebar{\Omega}_k=-\lambda_k\left(\widebar{\partial G}(U_k)+\widebar{U}_k^{-T}\widebar{\widebar{\partial G}}(U_k)^T\widebar{\widebar{U}}_k\right)\widebar{U}_k^{-1} \\
		\widebar{\widebar{\Omega}}_k=\lambda_k\widebar{U}_k^{-T}\widebar{\widebar{\partial G}}(U_k)^T.
		\end{cases}
		\end{equation*}
		
		$\bullet$ Form the matrix $\widetilde{\Omega}_k=\left[
		\begin{array}{cc}
		\widebar{\Omega}_k & \widebar{\widebar{\Omega}}_k \\
		-\widebar{\widebar{\Omega}}_k^T & \mathbb{O}
		\end{array}
		\right]$.
		
		$\bullet$ Compute $$U_{k+1}=P_{\nu_k^{-1}}^T\left(\mathbb{I}_n+\frac{1}{2}\wt{\Omega}_k\right)\left(\mathbb{I}_n-\frac{1}{2}\wt{\Omega}_k\right)^{-1}\wt{U}_k$$
		
		{\bf until} $U_{k+1}$ sufficiently minimizes $\wt{G}$.
		
	\end{itemize}
\end{framed}

An intrinsic way to construct an update for the steepest descent algorithm is to use a geodesic-like update. Using $QR$-decomposition, this has been constructed in \cite{edelman}. A quasi-geodesic update has been introduced in \cite{nishimori} and \cite{wen} for computational efficiency. An interesting retraction and its associated quasi-geodesic curves have been constructed in \cite{krakowski} in relation to interpolation problems on Stiefel manifolds.

An extrinsic method to update the algorithm is using projection-like retraction. For computational reasons various projection-like retraction updates have been constructed in \cite{absil-mahony-sepulchre-1}, \cite{manton}, \cite{bo}.
\medskip

{\bf Brockett cost function case.} For the case I, we consider the following particular cost function on $St^4_2$
$$G(U)=\mu_1 {\bf u}_1^TA{\bf u}_1+\mu_2 {\bf u}_2^TA{\bf u}_2,$$
where $\mu_1=1$, $\mu_2=2$, and $A=\text{diag}\,(1,2,3,4)$. The cost function being quadratic it is invariant under the sign change of the vectors that give the columns of the matrix $U$, but it is not invariant under the order of these column vectors. The set of critical points is given by:
\begin{itemize}[leftmargin=0.5cm]
\item four critical points generated by $[{\bf e}_2,{\bf e_1}]$ (i.e., $[{\bf e}_2,{\bf e_1}]$, $[-{\bf e}_2,{\bf e_1}]$, $[{\bf e}_2,-{\bf e_1}]$, and $[-{\bf e}_2,-{\bf e_1}]$) with the value of the cost function equals 4, which is a global minimum.
\item eight critical points generated by $[{\bf e}_1,{\bf e_2}]$ and $[{\bf e}_3,{\bf e_1}]$ with the value of the cost function equals 5.
\item four critical points generated by $[{\bf e}_4,{\bf e_1}]$ with the value of the cost function equals 6.
\item eight critical points generated by $[{\bf e}_1,{\bf e_3}]$ and $[{\bf e}_3,{\bf e_2}]$ with the value of the cost function equals 7.
\item eight critical points generated by $[{\bf e}_2,{\bf e_3}]$ and $[{\bf e}_4,{\bf e_2}]$ with the value of the cost function equals 8.
\item four critical points generated by $[{\bf e}_1,{\bf e_4}]$ with the value of the cost function equals 9.
\item eight critical points generated by $[{\bf e}_2,{\bf e_4}]$ and $[{\bf e}_4,{\bf e_3}]$ with the value of the cost function equals 10.
\item four critical points generated by $[{\bf e}_3,{\bf e_4}]$ with the value of the cost function equals 11, which is a global maximum.
\end{itemize}

For the case I, we have run the algorithm for some initial points and we show the convergence of the sequence of iterations toward the corresponding critical points.
\smallskip

\small{\begin{tabular}{|c|c|c|c|}
\hline 
& & & \\
$\boldsymbol{U_0}$ & $\boldsymbol{U_{300}}$ & \pbox{5cm}{
 $\boldsymbol{\simeq} \boldsymbol{U_{cr}}$ \\ ({\bf critical point})}  & $\boldsymbol{G(U_{cr})}$ \\
\hline & & & \\
$\left[\begin{array}{cc}
0&\displaystyle\frac{\sqrt{2}}{2}\\
-\displaystyle\frac{\sqrt{2}}{2}&0\\
0&-\displaystyle\frac{\sqrt{2}}{2}\\
-\displaystyle\frac{\sqrt{2}}{2}&0 
\end{array}\right]$ & $\left[\begin{array}{cc}
0&1.0000000\\
-1.0000000&0\\
0&-7.8365183\cdot 10^{-171}\\
-6.1260222\cdot 10^{-166}&0 
\end{array}\right]$ & 
$[-{\bf e}_2,{\bf e}_1]$ & 4 \\
 & & & \\
 \hline & & & \\
$\left[\begin{array}{cc}
0&\displaystyle\frac{\sqrt{3}}{3}\\[6pt]
-\displaystyle\frac{\sqrt{2}}{2}&\displaystyle\frac{\sqrt{3}}{3}\\[6pt]
0&0\\[6pt]
-\displaystyle\frac{\sqrt{2}}{2}&-\displaystyle\frac{\sqrt{3}}{3} 
\end{array}\right]$ & $\left[\begin{array}{cc}
-0.00021656&-0.99999998\\
-1.0000000&0.00021656\\
0&0\\
-1.8444858\cdot 10^{-10}&3.6515800\cdot 10^{-14}  
\end{array}\right]$ & 
$[-{\bf e}_2,-{\bf e}_1]$ &  4 \\
 & & & \\
 \hline & & & \\
$\left[\begin{array}{cc}
\displaystyle\frac{\sqrt{3}}{3}&-\displaystyle\frac{\sqrt{2}}{2}\\[8pt]
0&0\\[8pt]
-\displaystyle\frac{\sqrt{3}}{3}&-\displaystyle\frac{\sqrt{2}}{2}\\[8pt]
\displaystyle\frac{\sqrt{3}}{3}&0
\end{array}\right]$ & $\left[\begin{array}{cc}
-1.4227613\cdot 10^{-13}&-1.0000000\\
0&0\\
-1.0000000&1.4227614\cdot 10^{-13}\\
-1.7746315\cdot 10^{-14}&-1.9382582\cdot 10^{-14} 
\end{array}\right]$ & 
$[-{\bf e}_3,-{\bf e}_1]$ & 5 \\
 & & & \\
\hline
\end{tabular}}
\medskip

For the case II, when $\mu_1=\mu_2=1$ and the same matrix $A$, we obtain continuous families of critical points.\\
Starting from the initial point
$$U_0=\left[\begin{array}{cc}
\displaystyle\frac{1}{2}&0\\[8pt]
\displaystyle\frac{1}{2}&-\displaystyle\frac{\sqrt{2}}{2}\\[8pt]
-\displaystyle\frac{1}{2}&0\\[8pt]
-\displaystyle\frac{1}{2}&-\displaystyle\frac{\sqrt{2}}{2} 
\end{array}\right]
$$
the algorithm goes after 300 iterations to
$${\small U_{300}=\left[\begin{array}{cc}
0.97862435&0.20565599\\
0.20565598&-0.97862437\\
8.8491189\cdot 10^{-11}&-4.4716231\cdot 10^{-11}\\
9.2851360\cdot 10^{-12}&-4.6922465\cdot 10^{-12} 
\end{array}\right] \simeq \left[\begin{array}{cc}
0.9786243&0.2056559\\
0.2056559&-0.9786243\\
0&0\\
0&0 
\end{array}\right],}
$$
which is a rotation of the frame $\{{\bf e}_1,{\bf e}_2\}$ with an angle $\theta \simeq 0.207$ radians and the value of the cost function equals 3, which is a global minimum.\vspace{0.5cm}

\noindent {\bf Acknowledgment.} This work was supported by a grant of Ministery of Research and Innovation, CNCS - UEFISCDI, project number PN-III-P4-ID-PCE-2016-0165, within PNCDI III.


\begin{thebibliography}{99}


\bibitem{absil-mahony-sepulchre-1} {\bf P.A. Absil, R. Mahony, R. Sepulchre}, {\it Optimization Algorithms on Matrix Manifolds}, Princeton University Press, 2008.
\bibitem{adler}{\bf R.L. Adler, J.-P. Dedieu, J.Y. Margulies, M. Martens, M. Shub}, {\it Newton’s method
on Riemannian manifolds and a geometric model for the human spine}, IMA J. Numer. Anal., Vol. 22 (2002), pp. 359-390.
\bibitem{balogh}{\bf J. Balogh, T. Csendes, T. Rapcs\'{a}k}, {\it Some Global Optimization Problems on Stiefel Manifolds}, Journal of Global Optimization, Vol. 30, Issue 1 (2004), pp. 91-101.
\bibitem{birtea-comanescu} {\bf P. Birtea, D. Com\u anescu}, {\it Geometric dissipation for dynamical systems}, Comm. Math. Phys., Vol. 316, Issue 2 (2012), pp. 375-394.
\bibitem{Birtea-Comanescu-Hessian}{\bf  P. Birtea, D. Com\u anescu}, {\it Hessian Operators on Constraint Manifolds}, J. Nonlinear Science, Vol. 25, Issue 6 (2015), pp. 1285-1305.
\bibitem{birtea-comanescu-5-electron} {\bf P. Birtea, D. Com\u anescu}, {\it Newton Algorithm on Constraint Manifolds and the 5-Electron Thomson Problem}, J. Optim. Theor. Appl.,  Vol. 173, Issue 2 (2017), pp. 563-583.
\bibitem{bo} {\bf Bo Jiang, Yu-Hong Dai}, {\it A framework of constraint preserving update schemes for optimization on Stiefel manifold}, Math. Program., Ser. A, Vol. 153, Issue 2 (2015), pp. 535-575.
\bibitem{bolla}{\bf M. Bolla, G. Michaletzky, G. Tusn\'{a}dy, M. Ziermann}, {\it Extrema of Sums of Heterogeneous Quadratic Forms}, Linear Algebra and its Applications, Vol. 269, Issues 1–-3 (1998), pp. 331-365.
\bibitem{boyd} {\bf S. Boyd, N. Parikh, E. Chu, B. Peleato, J. Eckstein}, {\it Distributed Optimization and Statistical Learning via the Alternating Direction Method of Multipliers}, Foundations and Trends in Machine Learning, Vol. 3, Issue 1 (2010), pp. 1-122.
\bibitem{chen} {\bf Caihua Chen, Bingsheng He, Yinyu Ye, Xiaoming Yuan}, {\it The direct extension of ADMM for multi-block convex
minimization problems is not necessarily convergent}, Math. Program., Ser. A, Vol. 155, Issues 1-2 (2016), pp. 57-79.
\bibitem{chu} {\bf M.T. Chu, N.T. Trendafilov}, {\it The orthogonally constrained regression revisited},  J. Comput. and Graphical Statistics, Vol. 10, Issue 4 (2001), pp. 746-771. 
\bibitem{edelman}{\bf A. Edelman, T. A. Arias, S. T. Smith},{\it The
geometry of algorithms with orthogonality constraints}, SIAM J. Matrix Anal. Appl., Vol. 20, Issue 2 (1998), pp. 303-353.
\bibitem{elden}{\bf L. Eld\'{e}n, H. Park}, {\it A Procrustes problem on the Stiefel manifold}, Numer. Math., Vol. 82 (1999), pp. 599-619.
\bibitem{fraikin} {\bf C. Fraikin, K. H\"uper, P. Van Dooren}, {\it Optimization over the Stiefel
manifold}, Proc. in Appl. Math. Mech., Vol. 7, Issue 1 (2007).
\bibitem{glowinski}{\bf R. Glowinski}, {\it On alternating direction methods of multipliers: a historical perspective}. In: W. Fitzgibbon,
Y.A. Kuznetsov, P. Neittaanmaki, O. Pironneau (eds.), Modeling, Simulation and Optimization for
Science and Technology, Computational Methods in Applied Sciences, Vol. 34, pp. 59-82, Springer,
Dordrecht (2014).
\bibitem{glowinski-marrocco}{\bf R. Glowinski, A. Marrocco}, {\it Sur l'approximation par \'{e}l\'{e}ments finis d'ordre un, et la
r\'{e}solution, par p\'{e}nalisation-dualit\'{e} d'une classe de probl\`{e}mes de Dirichlet non lin\'{e}aires}, Rev. Fran\c{c}aise Automat. Inf. Rech. Op\'{e}rationnelle, Vol. 9, Issue 2 (1975), pp. 41-76.
\bibitem{kanamori}{\bf T. Kanamori, A. Takeda}, {\it Non-convex Optimization on Stiefel Manifold and Applications to Machine Learning}, Neural Information Processing - 19th International Conference, ICONIP 2012, Doha, Qatar, Proceedings, Part I, pp. 109-116, 2012.
\bibitem{krakowski} {\bf K.A. Krakowski, L. Machado, F.S. Leite, J. Batista}, {\it A modified Casteljau algorithm to solve interpolation
problems on Stiefel manifolds}, Journal of Computational and Applied Mathematics, Vol. 311 (2017), pp. 84-99.
\bibitem{manton}{\bf  J.H. Manton}, {\it Optimization algorithms exploiting unitary constraints}, IEEE Trans. Signal Process., Vol. 50 (2002), pp. 635-650.
\bibitem{nishimori} {\bf Y. Nishimori, S. Akaho}, {\it Learning algorithms utilizing quasi-geodesic flows on the Stiefel manifold}, Neurocomputing, Vol. 67 (2005), pp. 106-135.
\bibitem{petersen-pedersen}{\bf K.B. Petersen, M.S. Pedersen}, {\it The Matrix Cookbook}, 2012. 
\bibitem{rapcsak}{\bf T. Rapcs\'{a}k}, {\it On minimization on Stiefel manifolds}, European Journal of Operational Research, Vol. 143 (2002), pp. 365-376.
\bibitem{rosen}{\bf J.B. Rosen}, {\it The Gradient Projection Method for Nonlinear Programming. Part II. Nonlinear Constraints}, Journal of the Society for Industrial and Applied Mathematics, Vol. 9, Issue 4 (1961), pp. 514-532.
\bibitem{shub}{\bf M. Shub}, {\it Some remarks on dynamical systems and numerical analysis}. In: Dynamical Systems and
Partial Differential Equations (Caracas, 1984), pp. 69-91. Univ. Simon Bolivar, Caracas (1986).
\bibitem{wen} {\bf Z. Wen, W. Yin}, {\it A feasible method for optimization with orthogonality constraints}, Math. Program., Ser. A, Vol. 142, Issue 1 (2013), pp. 397-434.
\bibitem{zhang}{\bf Y. Zhang}, {\it Recent advances in alternating direction methods: Theory and practice}. In: IPAM Workshop: Numerical Methods for Continuous Optimization. UCLA, Los Angeles (2010).


\end{thebibliography}
\end{document}